\documentclass{amsart}
\usepackage{amsfonts,amssymb,amscd,amsmath,enumerate,verbatim,calc}
\usepackage{xy}
\usepackage{mathtools}
\usepackage{tikz-cd}
\usepackage[mathscr]{euscript}

\newcommand{\CM}{Cohen-Macaulay}

\newcommand{\n}{\mathfrak{n} }
\newcommand{\m}{\mathfrak{m} }

\newcommand{\T}{T_{\mathcal{F}}}
\newcommand{\eT}{e_{\mathcal{F}}^T}
\newcommand{\anram}{analytically unramified}
\newcommand{\e}{e_{A}^T}

\newcommand{\rt}{\rightarrow}

\newcommand{\ov}{\overline}
\newcommand{\sub}{\subseteq}

\newcommand{\CMS}{\operatorname{\underline{CM}}}
\newcommand{\CMa}{\operatorname{CM}}
\newcommand{\modA}{\operatorname{mod}}
\newcommand{\redA}{\operatorname{red}}
\newcommand{\cx}{\operatorname{cx}}
\newcommand{\depth}{\operatorname{depth}}

\newcommand{\Tor}{\operatorname{Tor}}
\newcommand{\Hom}{\operatorname{Hom}}
\newcommand{\Ext}{\operatorname{Ext}}
\newcommand{\Syz}{\operatorname{Syz}}

\theoremstyle{plain}

\newtheorem{theorem}{Theorem}[section]
\newtheorem{corollary}[theorem]{Corollary}
\newtheorem{lemma}[theorem]{Lemma}
\newtheorem{proposition}[theorem]{Proposition}

\theoremstyle{definition}
\newtheorem{definition}[theorem]{Definition}
\newtheorem{remark}[theorem]{Remark}
\newtheorem{s}[theorem]{}
\newtheorem{example}[theorem]{Example}

\theoremstyle{remark}

\begin{document}
	\title[sub-functor of Ext]{A sub-functor for Ext and  Cohen-Macaulay associated graded modules with bounded multiplicity-II}
	\author{Ankit Mishra}
	\email{chandra.ankit412@gmail.com}
	
	\author{ Tony~J.~Puthenpurakal}
	\email{tputhen@math.iitb.ac.in}
	
	\address{Department of Mathematics, IIT Bombay, Powai, Mumbai 400 076}

	\date{\today}
\subjclass{Primary 13A30, 13C14; Secondary 13D40, 13D07}
\keywords{Associated graded rings and modules, Brauer-Thrall conjectures, strict complete intersections, Henselian rings, Ulrich modules, integral closure filtration}
		\begin{abstract}
		Let $(A,\mathfrak{m})$ be a \CM\ local ring, then notion of $T$-split sequence was introduced in part-1 of this paper for $\mathfrak{m}$-adic filtration with the help of numerical function $e^T_A$.   We have explored the relation between AR-sequences and $T$-split sequences. For a Gorenstein ring $(A,\mathfrak{m})$ we define a Hom-finite Krull-Remak-Schmidt category $\mathcal{D}_A$ as a quotient of the stable category $\CMS(A)$. This category preserves isomorphism, i.e. $M\cong N$ in $\mathcal{D}_A$ if and only if $M\cong N$ in $\CMS(A)$.This article has two objectives; first objective is to extend the notion of $T$-split sequence, and second objective is to explore function $e^T_A$ and $T$-split sequence.		
		When $(A,\mathfrak{m})$ be an \anram\ \CM\ local ring and $I$ be an $\mathfrak{m}$-primary ideal  then we extend the techniques in part-1 of this paper to the integral closure filtration with respect to $I$ and prove a version of Brauer-Thrall-II for a class of such rings.
	\end{abstract}
\maketitle
	\section{Introduction}
\noindent Dear Reader, \\
While reading this paper
it will be a good idea to keep a copy of part 1, \cite{Pu trn},
 of this paper with you. The notation used in this paper are
same as that of \cite{Pu trn}.\\
Let $(A,\m)$ be a \CM \ local ring of dimension $d \geq 1$ and let $\CMa(A)$ be the category of maximal \CM \ $A$-modules. In \cite{Pu trn}
second author has constructed  $T \colon \CMa(A)\times \CMa(A) \rt \modA(A)$, a sub-functor of $\Ext^1_A(-, -)$
as follows:
Let $M$ be an MCM  $A$-module. Set
$$ e^T_A(M) = \lim_{n \rt \infty} \frac{(d-1)!}{n^{d-1}}\ell\left(\Tor^A_1(M, \frac{A}{\m^{n+1}}) \right ). $$
This function arose in the second author's study of certain aspects of the  theory of Hilbert functions \cite{Pu0},\cite{PuMCM}.
Using  \cite[Theorem 18]{Pu0} we get that $e^T_A(M)$ is a finite number  and it is zero if and only if $M$ is free.
Let $s \colon 0 \rt N \rt E \rt M \rt 0$ be an exact sequence
of MCM $A$-modules. Then by \cite[2.6]{Pu_stable} we get that $e^T_A(E) \leq e^T_A(M) + e^T_A(N)$. Set $e^T(s) = e^T_A(M) + e^T_A(N) - e^T_A(E)$.
\begin{definition}
We say $s$ is $T$-split if $e^T_A(s) = 0$.
\end{definition}
\begin{definition}
 Let $M, N$ be MCM $A$-modules. Set
 \[
  T_A(M,N) = \{ s \mid s  \ \text{is   a $T$-split extension} \}.
 \]
\end{definition}
We proved \cite[1.4]{Pu trn},
\begin{theorem}(with notation as above)
  $T_A \colon \CMa(A)\times \CMa(A) \rt \modA(A)$ is  a sub-functor of $\Ext^1_A(-, -)$.
\end{theorem}
It is not clear from the definition whether $T_A(M,N)$ is non-zero. Theorem \cite[1.5]{Pu trn} shows that there are plenty of $T$-split extensions
if $\dim \Ext^1_A(M,N) > 0$. We proved \cite[1.5]{Pu trn}
\begin{theorem}\label{fl}
 Let $(A,\m)$ be a \CM \ local ring and let $M,N$ be MCM $A$-modules. Then
\[
 \Ext^1_A(M, N)/T_A(M, N) \quad \text{has finite length.}
\]
\end{theorem}
Note Theorem \ref{fl} has no content if $M$ is free on the punctured spectrum of $A$. One of our motivations of this paper was to investigate $T_A(M, N)$ when $M$ is free on the punctured spectrum of $A$.

Now assume $(A,\m)$ is Henselian and  $M$ is an indecomposable MCM $A$-module with $M_{\mathfrak{p}}$ is free for all $\mathfrak{p}\in $ Spec$^0(A)=$ Spec$(A)\setminus \{\mathfrak{m}\}$,  then
a fundamental short exact sequence known as the Auslander-Reiten (AR)-sequence ending at $M$ exists. For a good introduction to AR-sequences see \cite[Chapter 2]{Yoshino}.

The following result gives a large number of  examples of AR-sequences which are $T$-split.
\begin{theorem}\label{cx2}
	Let $(Q,\mathfrak{n})$ be a Henselian regular local ring and $\underline{f} = f_1,\ldots,f_c\in \mathfrak{n}^2$ a regular sequence.  Set $I=(f_1,\ldots,f_c)$ and $(A,\mathfrak{m})= (Q/I,\mathfrak{n}/I)$. Assume $\dim A = 1$. Let $M$ be an indecomposable MCM $A$-module with $\cx_A M\geq 2$. Assume $M$ is free on {\normalfont Spec}$^0(A)$. Set $M_n= \Syz_n^A(M)$, then for $n\gg 0$ the AR-sequences ending in $M_n$ are $T$-split.
\end{theorem}
For hypersurfaces defined by quadrics we prove:
\begin{theorem}
\label{quadric} Let $(Q,\mathfrak{n})$ be a Henselian regular local ring with algebraically closed residue field $k = Q/\mathfrak{n}$  and let $f \in \mathfrak{n}^2 \setminus \mathfrak{n}^3$. Assume the hypersurface $A  = Q/(f)$ is an isolated singularity.
Then all but finitely many AR-sequences in $A$ are $T$-split.
\end{theorem}

Theorems \ref{cx2} and \ref{quadric} shows that $T$-split sequences are abundant in general. However the following example is important:
\begin{example}\label{counter}
  There exists a complete hypersurface isolated singularity $A$ and an indecomposable MCM $A$-module $M$ such that $T_A(M, N) = 0$ for any MCM $A$-module $N$.
\end{example}

\s Now assume $A$ is Gorenstein. As observed in \cite{Pu_stable} the function $e^T_A(-)$ is infact a function on $\CMS(A)$ the stable category  of all MCM $A$-modules.
Let $M$ and $N$ be  MCM $A$-modules.
	It is well-known that we have a natural isomorphism $$\eta \colon \text{\underline{Hom}}_A(M,N)\cong \text{Ext}^1_A(\Omega^{-1}(M),N)$$
	 Let $T_A(\Omega^{-1}(M),N)$ denotes the set of all $T$-split sequences in Ext$^1_A(\Omega^{-1}(M),N)$. We denote $\eta^{-1}(T_A(\Omega^{-1}(M),N))$ by $\mathcal{R}(M,N)$. Then $\eta $ induces following isomorphism
	$$\frac{\text{\underline{Hom}}_A(M,N)}{\mathcal{R}(M,N)}\cong \frac{\text{Ext}^1_A(\Omega^{-1}(M),N)}{T_A(\Omega^{-1}(M),N)}.$$
Suprisingly
\begin{proposition}\label{relation}
	$\mathcal{R}$ is a relation on $\CMS(A)$.
\end{proposition}
Thus we may consider the quotient category $\mathcal{D}_A = \CMS(A)/\mathcal{R}$. Clearly $\mathcal{D}_A$ is a Hom-finite additive category. Surprisingly the following result holds
\begin{theorem}\label{krull-schmidt}
	Let $(A,\mathfrak{m})$ be a Henselian Gorenstein local ring and let $M$ and $N$ be MCM $A$-modules.  Then the following holds
\begin{enumerate}[\rm (1)]
  \item  $M\cong N$ in $\mathcal{D}_A$ if and only if $M\cong N$ in $\CMS(A)$.	
  \item $M$ is indecomposable in $\mathcal{D}_A$ if and only if $M$ is indecomposable in $\CMS(A)$
  \item $\mathcal{D}_A$ is a Krull-Remak-Schmidt (KRS) category.
\end{enumerate}
\end{theorem}

The main application of $T$-split sequences was to study  {\it Weak Brauer-Thrall-II} for associated graded modules for a large class of rings.
Note that in \cite{Pu trn} the concept was introduced only for $\mathfrak{m}$-adic filtration, but for general $I(\ne \mathfrak{m})$-adic filtrations that method will not work (see \cite[Remark 3.2]{Pu trn}).

In this article we extend the results in \cite{Pu trn} to a large family of filtrations.
Let $(A,\mathfrak{m})$ be an  analytically unramified \CM \ local ring of dimension $d \geq 1$ and let $I$ be an $\m$-primary ideal. Let $\mathcal{F}=\{I_n\}_{n \in \mathbb{Z}}$ where $I_1=\overline{I}$ and $I_n=\overline{I^{n}}$ for $n\gg 0$ be an $I$-admissible filtration. Let $M$ be an MCM $A$-module.
 \s \label{ankit} Set
$$ e_{\mathcal{F}}^T(M) = \lim_{n\to \infty}\frac{(d-1)!}{n^{d-1}}\ell(\Tor_1^A(M,A/{I_{n+1}})) $$
Then $e_{\mathcal{F}}^T(M) = 0$ if and only if $M$ is free (see \cite[Theorem 7.5]{Kadu}).
Let $M,N$ be maximal Cohen-Macaulay $A$-modules and  $\alpha\in $ Ext$_A^1(M,N)$. Let $\alpha$ be given by an extension $0\rt N\rt E\rt M\rt 0$, here $E$ is an maximal Cohen-Macaulay module.
 Now set
$$e_{\mathcal{F}}^T(\alpha)=e_{\mathcal{F}}^T(M)+e_{\mathcal{F}}^T(N)-e_{\mathcal{F}}^T(E).$$
It can be shown that $e_{\mathcal{F}}^T(\alpha) \geq 0$, see \ref{sub-add}.
\begin{definition}
	An extension $s\in \Ext_A^1(M,N)$ is $T_{\mathcal{F}}$-split if $e_{\mathcal{F}}^T(s)=0$.
\end{definition}
As before we can show that $\T(M,N)$ is a submodule of $\Ext^1_A(M,N)$ (see \ref{TsubmoduleExt}). Furthermore $\T \colon \CMa(-)\times \CMa(-) \rt \modA(A)$ is a sub-functor of $\Ext^1_A(-,-)$.,
 see \ref{Functor}.

\begin{theorem}\label{main-intro}
	Let $(A,\mathfrak{m})$ be an analytically unramified  Cohen-Macaulay local ring of dimension $d$ with one of the following conditions:
	\begin{enumerate}[ \rm(1)]
		\item 	 the residue field $k (= A/\mathfrak{m})$ is uncountable.
		\item   the  residue field $k$ is perfect field.
	\end{enumerate}
 Let $I$ be an  $\mathfrak{m}$-primary ideal and $\mathcal{F}=\{{I_n}\}_{n \in \mathbb{Z}}$ where $I_1=\overline{I}$ and $I_n= \overline{I^{n}}$ for $n\gg0$. Let $M, N$ be MCM  $A$-module then $\Ext_A^1(M,N)/T_{\mathcal{F}}(M,N)$ has finite length.	
\end{theorem}
Next we prove following theorem.
\begin{theorem}\label{BT-intro}
	Let $(A,\mathfrak{m})$ be a complete  reduced \CM\ local ring of dimension $d\geq1$ with either uncountable residue field or a perfect residue field. Let $I$ be an $\mathfrak{m}$-primary ideal. Set $R=A[[X_1,\ldots,X_m]]$, $J=(I,X_1,\ldots,X_m)$, $\mathcal{I}=\{\overline{I^n}\}_{n \in \mathbb{Z}}$, and $\mathcal{J}=\{\overline{J^n}\}_{n \in \mathbb{Z}}$. If $A$ has an MCM module $M$ with $G_{\mathcal{I}}(M)$ \CM. Then there exists $\{ E_n \}_{n \geq 1}$ indecomposable MCM $R$-modules with bounded multiplicity (with respect to $\mathcal{J}$) and having
$G_{\mathcal{J}}(E_n)$ \CM \ for all $n \geq 1$.
\end{theorem}

Let $e^T_\mathcal{F} \colon \Ext_A^1(M, N) \rt \mathbb{N}$ be the function defined by $\alpha \mapsto e^T_\mathcal{F}(\alpha)$. When $A$ has characteristic $p > 0$ then we can say more about this function.
If $V$ is a vector-space over a field $k$ then let $\mathbb{P}(V)$ denote the projective space determined by $V$.
\begin{theorem}\label{proj-intro}
(with hypotheses as in \ref{main-intro}) Further assume $A$ is of characteristic $p > 0$ and that $A$ contains a field $k \cong A/\m$. If $\Ext_A^1(M, N) \neq T_{\mathcal{F}}(M,N)$ then the function $e^T_\mathcal{F}$ factors as
$$ [\ov{e^T_A}] \colon \mathbb{P}(\Ext^1_A(M, N)/T_{\mathcal{F}}(M, N)) \rt \mathbb{N}\setminus{0}. $$
\end{theorem}

We now describe in brief the contents of this paper. In section two we discuss some preliminary results. In section three we introduce our function \ref{ankit} and discuss few of its properties
We also discuss in detail the base changes that we need to prove our results. In section four we prove Theorem \ref{main-intro}. In the next section we prove Theorem \ref{BT-intro}.
In section six we prove Theorem \ref{proj-intro}. In the next section we discuss our result on relation between T-split sequences and AR-sequences. In section eight we prove Theorem \ref{cx2}.
In the next section we  prove Theorem \ref{quadric} and construct Example \ref{counter}.
In section ten we prove Proposition \ref{relation} and Theorem \ref{krull-schmidt}.
\section{Preliminaries}

\s Let $(A,\mathfrak{m})$ be a Noetherian local ring and $I$ be an $\mathfrak{m}$-primary ideal. Then a filtration $\mathcal{F}=\{F_n\}_{n\in \mathbb{Z}}$ is said to be $I$-admissible filtration if
\begin{enumerate}
	\item $I^n\sub F_n$ for all $n$.
    \item $F_nF_m \subseteq F_{n + m}$ for all $n,m \in \mathbb{Z}$.
	\item $F_n = IF_{n-1}$ for $n\gg0$.
\end{enumerate}
 \begin{definition}
 	A Noetherian local ring $(A,\mathfrak{m})$ is said to be analytically unramified if its $\mathfrak{m}$-adic completion is reduced.
 \end{definition}
\s Let $\overline{\mathfrak{a}}$ denote integral closure of the ideal $\mathfrak{a}$. If $A$ is analytically unramified then from  a result of Rees \cite{Rees}, the integral closure filtration $\mathcal{F}=\{\overline{I^n}\}_{n \in \mathbb{Z}}$ is $I$-admissible.
\s Let $(A,\mathfrak{m})$ be a Noetherian local ring, $I$ an $\mathfrak{m}$-primary ideal and $\mathcal{F}=\{F_n\}_{n\in \mathbb{Z}}$ a $I$-admissible filtration. Let $M$ be a finite $A$-module with dimension $r$. Then the numerical function $H_{\mathcal{F}}(M,n)= \ell(M/{F_{n+1}}M)$ is known as the Hilbert function of $M$ with respect to $\mathcal{F}$. For large value of $n$, $H_{\mathcal{F}}(M,n)$ coincides with a polynomial $P_{\mathcal{F}}(M,n)$ of degree $r$; and this polynomial known as the Hilbert polynomial of $M$ with respect to $\mathcal{F}$. There exist unique integer $e_0^{\mathcal{F}}(M),e_1^{\mathcal{F}}(M),\ldots,e_r^{\mathcal{F}}(M)$ such that Hilbert polynomial of $M$ with respect to $\mathcal{F}$ can be written as
$$P_{\mathcal{F}}(M,n)= \sum_{i=0}^{r}(-1)^ie_i^{\mathcal{F}}(M)\binom{n+r-i}{r-i}.$$
These integers $e_0^{\mathcal{F}}(M),e_1^{\mathcal{F}}(M),\ldots,e_r^{\mathcal{F}}(M)$ are known as the Hilbert coefficients of $M$ with respect to $\mathcal{F}$. In case of $\mathfrak{m}$-adic and  $I$-adic filtrations these coefficients will be denoted as $e_i(M)$ and  $e_i^I(M)$  for $i=1,\ldots,r$ respectively.
\s \label{superficial} Let $(A,\mathfrak{m})$ be a Noetherian local ring and $I$ be an $\mathfrak{m}$-primary ideal. Let $\mathcal{F}=\{F_n\}_{n\in \mathbb{Z}}$ be an $I$-admissible filtration and $M$ an $A$-module with positive dimension. Then an element $x\in F_1\setminus F_2$ is said to be $\mathcal{F}$-superficial element for $M$ if there exists $c\in \mathbb{N}$ such that for all $n\geq c$, $$(F_{n+1}M:_Mx)\cap F_cM= F_nM.$$

The following facts are well known:
\begin{enumerate}
	\item If $k= A/\mathfrak{m}$ is infinite, then $\mathcal{F}$-superficial elements for $M$ exist.
	\item If  $\depth M>0$ then every $\mathcal{F}$-superficial element for $M$ is also $M$-regular.
	\item If $x$ is $\mathcal{F}$-superficial element for $M$ and $\depth M > 0$ then $(F_{n+1}M:_Mx)= F_nM$ for $n\gg 0$.
     \item If $x$ is $\mathcal{F}$-superficial element for $M$ and $\depth M > 0$ then $e_i^{\ov{\mathcal{F}}} (M/xM) = e_i^{\mathcal{F}}(M)$ for $i = 0,1,\ldots, \dim M - 1$ (here $\ov{\mathcal{F}}$
         is the obvious quotient filtration of $\mathcal{F}$).
\end{enumerate}
\s A sequence $\underline{x}=x_1,\ldots,x_r$ with $r\leq \dim M$ is said to be $\mathcal{F}$-superficial sequence for $M$ if $x_1$ is $\mathcal{F}$-superficial element for $M$ and $x_i$ is  $\mathcal{F}/(x_1,\ldots,x_{i-1})$-superficial element for $M/(x_1,\ldots,x_{i-1})M$ for all $i\leq r$.

\section{The case when $A$ is analytically unramified}
 Let $(A,\mathfrak{m})$ be an \anram\ \CM \ local ring with $\dim A = d\geq 1$, $I$ an $\mathfrak{m}$-primary ideal.
We are primarily interested in the integral closure filtration of $I$. However to prove our results we need the following class of $I$-admissible filtrations $\mathcal{F}=\{{I_n}\}_{n \in \mathbb{Z}}$ where $I_1=\overline{I}$ and $I_n=\overline{I^n}$ for $n\gg 0$.
 Let $M$ be an MCM $A$-module.

\s The numerical function
$$n\longmapsto \ell(\Tor_1^A(M,A/{{I_{n+1}}}))$$
is polynomial type; that is there is a polynomial $t^A_{\mathcal{F}}(M,z)$ which coincides with this numerical function for $n\gg 0$.\\
If $M$ is non free MCM $A$-module then $\deg t^A_{\mathcal{F}}(M,z)=d-1$ (see \cite[Theorem 7.5]{Kadu}). Note that normalised leading coefficient of $t^A_{\mathcal{F}}(M,z)$ is $e_1^{\mathcal{F}}(A)\mu(M)-e_1^{\mathcal{F}}(M)-e_1^{\mathcal{F}}(\text{Syz}_1^A(M))$.

\s \label{eT-formula} Set
\begin{align*}
e_{\mathcal{F}}^T(M)&=\lim_{n\to \infty}\frac{(d-1)!}{n^{d-1}}\ell(\Tor_1^A(M,A/{I_{n+1}}))\\
&=e_1^{\mathcal{F}}(A)\mu(M)-e_1^{\mathcal{F}}(M)-e_1^{\mathcal{F}}(\text{Syz}_1^A(M))
\end{align*}

\s\label{BS} \textbf{Base-change:}
We need to do several base changes in our arguments.

(I) We first discuss the general setup:
Let $\psi \colon (A,\m)\rt (B,\n)$ be a flat map such that $B$ is also a  \CM \ local ring and $\m B =\n$.
If $M$ is an $A$-module set $M_B= M\otimes_A B$. If $\mathcal{F}=\{{I_n}\}_{n \in \mathbb{Z}}$ is an $I$-admissible filtration then set  $\mathcal{F}_B=\{{I_n B}\}_{n \in \mathbb{Z}}$.
 Then
\begin{enumerate}
	\item $\ell(N)=\ell(N_B)$ for any finite length $A$-module $N$.
    \item $\mathcal{F}_B$ is an  $IB$-admissible filtration.
	\item  $\dim M=  \dim M_B$ and  $\depth M =  \depth M_B$. In particular, $M$ is an MCM $A$-module if and only if $M_B$ is MCM $B$-module.
\item Syz$^A_i(M)\otimes_A B\cong \text{Syz}^B_i(M_B)$ for all $i \geq 0$.
	\item $e_i^\mathcal{F}(M)= e_i^{\mathcal{F}_B}(M_B)$ for all $i$.
	\item If $\psi$ is regular and $\mathfrak{a}$ is integrally closed $\mathfrak{m}$-primary ideal in $A$ then $\mathfrak{a}B$ is integrally closed in $B$ (for instance see \cite[2.2(7)]{Heinzer et al.}).
\end{enumerate}
 (II) Assume $A$ is analytically unramified \CM \ local ring and $\mathcal{F}=\{{I_n}\}_{n \in \mathbb{Z}}$ is an $I$-admissible filtration with $I$, $\m$-primary and furthermore $I_1=\overline{I}$ and $I_n=\overline{I^n}$ for $n\gg 0$. We need to base changes as above where $\mathcal{F}_B$ has the property that $I_nB = I^nB = \ov{I^nB}$ whenever $I_n = \ov{I^n}$.
 Note this automatically forces $B$ to be analytically unramified.
 The specific base changes we do are the following:\\
 (i) $B = \widehat{A}$ the completion of $A$. Note that if $J$ is an $\m$-primary integrally closed ideal then $J\widehat{A}$ is also integrally closed, cf., \cite[9.1.1]{HS}.

 (ii) If $A$ has a finite residue field then we can consider the extension $B = A[X]_{\m A[X]}$.  The residue field of $B$ is $k(X)$ which is infinite. Note that if $J$ is an $\m$-primary integrally closed ideal then $J\widehat{A}$ is also integrally closed, cf., \cite[8.4.2]{HS}.

 (iii) Assume $\dim A \geq 2$. Even if $A$ has infinite residue field there might not exist $\mathcal{F}$-superficial element $x$ such that $A/(x)$ is analytically unramified. However a suitable extension $B$ has this property. To see this we first observe two facts. \\
  Let $\mathcal{E}$ be a countable set of MCM  of $A$-modules.  Assume that the residue field $k$ of $A$ is uncountable if $\mathcal{E}$ is an infinite set. Otherwise $k$ is infinite. \\
  (a) There exist $\underline{x}=x_1,\ldots,x_d\in \overline{I}$ such that $\underline{x}$ is $\mathcal{F}$-superficial for each $N \in \mathcal{E}$.  This result is well-known (for instance see \cite[Lemma 2.2]{Pu_GrowthSyz}).\\
   (b)  There exists a generating set $r_1,\ldots,r_t$ of $I$ such that for each $i$, $r_i$ is $I$-superficial and $\mathcal{F}$-superficial element for each $N \in \mathcal{E}$ (see (a) and
     \cite[Lemma 7.3]{Kadu}).

(\cite{Ciup_genr_elm},\cite[Lemma 7.4, Theorem 7.5]{Kadu}) Choose $r_1,\ldots,r_t$ as in (b). Now consider following flat extension of rings $$A\rt \hat{A}\rt B=\hat{A}[X_1,\ldots,X_t]_{\mathfrak{m}\hat{A}[X_1,\ldots,X_t]}.$$
Let  $\zeta= r_1X_1+\ldots+r_tX_t$. Set $C = B/\zeta B$ and $\mathcal{F}_C=\{I_n C\}$.
For $N \in \mathcal{E}$, set $N_B = N\otimes_A B$.
 Then we have

  \begin{enumerate}
	\item $B$ is \anram\ \CM\ local ring of dimension $d$.
	\item $N_B$ is MCM $B$-module for each $N \in \mathcal{E}$.
      \item If $J$ is a integrally closed $\m$-primary ideal in $A$ then $JB$ a integrally closed $\n$-primary ideal in $B$.
	\item $I_1C = \overline{I}C = \overline{IC}$.
	\item $I_nC=\overline{I^nC}$ for all $n\gg 0$.
	\item $C$ is \anram\ \CM\ local ring of dimension $d-1$.
	\item $\zeta$ is $\mathcal{F}_B$-superficial for each  $N_B$ (here $N \in \mathcal{E}$).
\end{enumerate}

(iv) For some of our arguments we need the residue field of $A$ to be uncountable. If  $k$ is finite or countably infinite do the following: First complete $A$. By (i) this is possible. So we may assume $A$ is complete.
	
	Consider extension $\phi : A\longrightarrow A[[X]]_{\mathfrak{m}A[[X]]}=(B,\mathfrak{n})$. Set $B_0=B\otimes_Ak=B/\mathfrak{m}B$. So $B_0=B/\mathfrak{n}=k((X))$ is uncountable. As $k$ is perfect we get  $k((X))$ is $0$-smooth over $k$, see \cite[28.7]{Matsu}.  Using \cite[28.10]{Matsu} we get  $B$ is $\mathfrak{n} (=\mathfrak{m}B)$-smooth. This implies $\phi$ is regular (see \cite[Theorem]{Andre}).
	
	By I(6) if $\mathfrak{a}$ is an integrally closed $\mathfrak{m}$-primary ideal in $A$ then $\mathfrak{a}  B$ is integrally closed in $B$.
Thus   ${I_n}B=\overline{I^nB}$ whenever $I_n = \ov{I^n}$.

\begin{definition}\label{good-def}
We say  a flat extension $\psi \colon (A,\m) \rt (B,\n)$ with $\m B = \n$ \emph{behaves well} with respect to integral closure if for any integrally closed $\mathfrak{m}$-primary ideal $\mathfrak{a}$  in $A$ the ideal $\mathfrak{a}  B$ is integrally closed in $B$.
\end{definition}

We need the following result:
\begin{proposition}\label{good}
Let $(A,\mathfrak{m})$ be \anram\ \CM \ local ring with $\dim A = d\geq 1$, $I$ an $\mathfrak{m}$-primary ideal.
  Let $\mathcal{F}=\{{I_n}\}_{n \in \mathbb{Z}}$  be a $I$-admissible filtration where $I_1=\overline{I}$ and $I_n=\overline{I^n}$ for $n\gg 0$.
  \begin{enumerate}[\rm (1)]
    \item Let $(B,\n)$ be a flat extension of $A$ which behaves well with respect to integral closure. Set $\mathcal{F}_B = \{{I_nB}\}_{n \in \mathbb{Z}}$.
    Then for any MCM $A$-module $M$ we have $e^T_\mathcal{F}(M) = e^T_{\mathcal{F}_B}(M_B)$.
    \item Let $\dim A \geq 2$.  and the residue field of $A$ is infinite. Let $\mathcal{V}$ be any countable set of MCM $A$-modules (containing $A$).
     Assume $k = A/\m$ is uncountable if $\mathcal{V}$ is infinite otherwise $k$ is infinite.
     Then there exists a flat extension $B$ of $A$  which behaves well with respect to integral closure such that there exist $\zeta \in IB$ which is $\mathcal{F}_B$-superficial with respect to  each  $N_B$ (for all $N \in \mathcal{V}$). Furthermore if $C = B/\zeta B$ then $C$ is analytically unramified with  $IC = \ov{IC}$ and $I_nC = \ov{I^nC}$ for all $n \gg 0$. Set $N_C = N\otimes_A C$.
         Furthermore $e^T_{\mathcal{F}}(N) = e^T_{\mathcal{F}_C}(N_C)$ for each $N \in \mathcal{V}$.
  \end{enumerate}
\end{proposition}
\begin{proof}
  (1) This follows from \ref{BS}(I).

  (2) Set $\mathcal{E}  = \{ \Syz^A_i(N) \colon i = 0, 1 \  \text{and} \ N \in \mathcal{V}\}$. Then $\mathcal{E}$ is a countable  set and is finite if $\mathcal{V}$ is. Now do the construction in \ref{BS}(I)(iii) and use \ref{superficial}(4) and \ref{eT-formula} to conclude.
\end{proof}
The following lemma follows from \cite[Theorem 2.6]{Pu_stable}, but here we give a short proof (Similar proof also works for $\e()$):
\begin{lemma}\label{sub-add}
	Let $\alpha : 0\rt N\rt E\rt M\rt 0$ be an exact sequence of MCM $A$-modules. Then $\eT(E)\leq \eT(M)+\eT(N)$.
\end{lemma}
\begin{proof}
	Consider the long exact sequence of $\alpha\otimes_A A/{{I_{n+1}}}$. We get
	$$\ldots \rt \Tor_1^A(N,A/{{I_{n+1}}})\rt \Tor_1^A(E,A/{{I_{n+1}}})\rt \Tor_1^A(M,A/{{I_{n+1}}})\rt \ldots $$
	So, $\ell(\Tor^A_1(E,A/{I_{n+1}}))\leq \ell(\Tor^A_1(M,A/{I_{n+1}})) + \ell(\Tor^A_1(N,A/{I_{n+1}})) $.
	Now from the definition of $\eT(-)$, required inequality follows.
\end{proof}

\s Let $M,N$ be maximal Cohen-Macaulay $A$-modules and  $\alpha\in \Ext_A^1(M,N)$. Let $\alpha$ be given by an extension $0\rt N\rt E\rt M\rt 0$, here $E$ is an maximal Cohen-Macaulay module. Now set
$$e_{\mathcal{F}}^T(\alpha)=e_{\mathcal{F}}^T(M)+e_{\mathcal{F}}^T(N)-e_{\mathcal{F}}^T(E).$$
\s Let $\alpha_1,\alpha_2\in \Ext^1_A(M,N)$. Suppose $\alpha_i$ can be given by
$0\rt N\rt E_i\rt M\rt 0$ for $i=1,2$. If $\alpha_1$ and $\alpha_2$ are equivalent then $E_1\cong E_2$. So $\eT(\alpha_1)=\eT(\alpha_2)$. This implies $\eT(\alpha)$ is well defined.

Note that $\eT(\alpha)\geq 0$.
\begin{definition}
	An extension $s\in $ Ext$_A^1(M,N)$ is $T_{\mathcal{F}}$-split if $e_{\mathcal{F}}^T(s)=0$.
\end{definition}

\begin{definition}
	Let $M,N$ be maximal Cohen-Macaulay $A$-modules. Set $$T_{{\mathcal{F}},A}(M,N)=\{s| s \ \text{is a } T_{\mathcal{F}}\text{-split extension}\}.$$
	Note that if the choice of the ring $A$ is unambiguous from the context, we denote this set as $\T(M,N)$.
\end{definition}

We will need the following two results:
\begin{lemma}\label{lsplit}
	Let $(A,\mathfrak{m})$ be an \anram\ \CM\ local ring of dimension $d\geq 1$ and let $M, N, N_1, E, E_1$ be  MCM $A$-modules. Suppose we have a commutative diagram
	
	\begin{center}

	\begin{tikzcd}
	\alpha: 0 \arrow[r, "" ]  &N \arrow[r, ""] \arrow[d ] &E \arrow[r, ""]\arrow[d ] &M \arrow[r, ""] \arrow[d, "1_M" ] &0\\
	\beta: 0 \arrow[r, ""] &N_1 \arrow[r ] &E_1 \arrow[r, ""] &M \arrow[r, ""] &0
	\end{tikzcd}
\end{center}
 If $\alpha$ is $\T$ -split, then $\beta$ is also $\T$-split.
\end{lemma}
\begin{proof}
	If $\dim A = 1$ then we can give an argument similar to \cite[Proposition 3.8]{Pu trn}.  Now assume $d = \dim A \geq 2$ and the result has been proved for all analytically unramified rings of dimension $d -1$.  If the residue field of $A$ is finite then use \ref{BS}II.(ii).  So we may assume $A/\m$ is infinite. Using \ref{good} we may assume that (after going to a flat extension) there exists $\zeta \in I$
such that

(i) $\zeta$ is
$\mathcal{F}$-superficial with respect to $A\oplus U \oplus \Syz^A_1(U)$  for each $U$ in the above diagram.

(ii) $C = A/\zeta A$ is analytically unramified with  $IC = \ov{IC}$ and $I_nC = \ov{I^nC}$ for all $n \gg 0$.

(iii) $e^T_{\mathcal{F}_C}(U/\zeta U) = e^T_{\mathcal{F}}(U)$ for each $U$ in the above diagram.

Notice  $\alpha\otimes C$ and $\beta\otimes C$ are exact. For an $A$-module $V$ set $\ov{V} = V/\zeta V$. So we have a diagram
\begin{center}
\begin{tikzcd}
	\alpha\otimes C: 0 \arrow[r, "" ]  &\ov{N} \arrow[r, ""] \arrow[d ] &\ov{E} \arrow[r, ""]\arrow[d ] &\ov{M} \arrow[r, ""] \arrow[d, "1_{\ov{M}}" ] &0\\
	\beta\otimes C: 0 \arrow[r, ""] &\ov{N_1} \arrow[r ] &\ov{E_1} \arrow[r, ""] &\ov{M} \arrow[r, ""] &0
	\end{tikzcd}
\end{center}
Note $\alpha\otimes C$ is $T_{\mathcal{F}_C}$ -split. By our induction hypotheses $\beta\otimes C$ is $T_{\mathcal{F}_C}$ -split. By our construction it follows that $\beta$ is also $\T$-split.
	\end{proof}
\begin{lemma}\label{usplit}
		Let $(A,\mathfrak{m})$ be an \anram\ \CM\ local ring of dimension $d\geq 1$ and let $M, M_1, N, E, E_1$ be  MCM $A$-modules. Suppose we have a commutative diagram
	
	\begin{center}
	\begin{tikzcd}
	\alpha: 0 \arrow[r, "" ]  &N \arrow[r, ""] \arrow[d, "1_N" ] &E \arrow[r, ""]\arrow[d ] &M \arrow[r, ""] \arrow[d ] &0\\
	\beta: 0 \arrow[r, ""] &N \arrow[r ] &E_1 \arrow[r, ""] &M_1 \arrow[r, ""] &0
	\end{tikzcd}
\end{center}
	 If $\beta$ is $\T$ -split, then $\alpha$ is also $\T$-split.
\end{lemma}
\begin{proof}
	This is dual to \ref{lsplit}.
\end{proof}

\section{$\T$-split sequences}
In this section we prove our results regarding $\T$.
\begin{theorem}\label{TsubmoduleExt}
	Let $(A,\mathfrak{m})$ be an analytically unramified Cohen-Macaulay local ring of dimension $d$. Let $I$ be an $\mathfrak{m}$-primary ideal and and $\mathcal{F}=\{{I_n}\}$ where $I_1=\overline{I}$ and $I_n= \overline{I^{n}}$ for $n\gg0$. Let $M, N$ be MCM  $A$-module, then $\T(M,N)$ is a submodule of  $\Ext^1_A(M,N)$.
\end{theorem}
\begin{proof}
	Let $\alpha : 0\rt N\rt E\rt M\rt 0$ be a $\T$-split extension and $r\in A$, then we can define $r\alpha$
		\begin{center}
		\begin{tikzcd}
		\alpha: 0 \arrow[r, "" ]  &N \arrow[r, ""] \arrow[d, "r" ] &E \arrow[r, ""]\arrow[d ] &M \arrow[r, ""] \arrow[d, "1_M" ] &0\\
		r\alpha: 0 \arrow[r, ""] &N \arrow[r ] &E_1 \arrow[r, ""] &M \arrow[r, ""] &0
		\end{tikzcd}
	\end{center}
Note that first square is push-out diagram. Since $\alpha $ is $\T$-split, this implies $r\alpha $ is also $\T$-split (see \ref{lsplit}).

Let $\alpha: 0\rt N\rt E\rt M\rt 0$ and $\alpha': : 0\rt N\rt E'\rt M\rt 0$ be two $\T$-split extensions. We want to show $\alpha +\alpha'$ is also $\T$-split.
Note that the addition operation on $\Ext^1_A(M,N)$ is Bear sum, that is $\alpha +\alpha':= (\nabla(\alpha \oplus\alpha'))\Delta$.

	Since $\alpha$ and $\alpha'$ are $\T$-split this implies $\alpha \oplus\alpha': 0\rt N\oplus N\rt E\oplus E'\rt M\oplus M\rt 0$ also $\T$-split. Consider following diagram
	 \begin{center}
	 	\begin{tikzcd}
	 	(\alpha\oplus \alpha'): 0 \arrow[r, "" ]  &N\oplus N \arrow[r, ""] \arrow[d, "\nabla" ] &E\oplus E' \arrow[r, ""]\arrow[d ] &M\oplus M \arrow[r, ""] \arrow[d, "1_M" ] &0\\
	 	\nabla(\alpha\oplus \alpha'): 0 \arrow[r, ""] &N \arrow[r ] &E_1 \arrow[r, ""] &M\oplus M \arrow[r, ""] &0
	 	\end{tikzcd}
	 \end{center}
  Note that first square is is pushout diagram.
 From \ref{lsplit}, $\nabla(\alpha\oplus \alpha')$ is $\T$-split.
 Now consider the diagram
  \begin{center}
 	\begin{tikzcd}
 	\nabla(\alpha\oplus \alpha')\Delta: 0 \arrow[r, "" ]  &N \arrow[r, ""] \arrow[d, "1_N" ] &E_2 \arrow[r, ""]\arrow[d ] &M \arrow[r, ""] \arrow[d, "\Delta" ] &0\\
 	\nabla(\alpha\oplus \alpha'): 0 \arrow[r, ""] &N \arrow[r ] &E_1 \arrow[r, ""] &M\oplus M \arrow[r, ""] &0
 	\end{tikzcd}
 \end{center}
Here second square is pullback diagram. Now from \ref{usplit}, $\alpha + \alpha=(\nabla(\alpha\oplus \alpha'))\Delta$ is $\T$-split.
	\end{proof}

We now show
\begin{theorem}\label{Functor}
(with hypotheses as in \ref{TsubmoduleExt}) $\T \colon \CMa(-) \times \CMa(-) \rt \modA(A)$ is a functor.
\end{theorem}
\begin{proof}
  This is similar to \cite[3.13]{Pu trn}. We have to use Theorem \ref{TsubmoduleExt} and Lemmas \ref{lsplit}, \ref{usplit}.
\end{proof}
The following is one of the main result of our paper.
\begin{theorem}\label{main1}
	Let $(A,\mathfrak{m})$ be an analytically unramified Cohen-Macaulay local ring of dimension $d$ with uncountable residue field. Let $I$ be an  $\mathfrak{m}$-primary ideal and $\mathcal{F}=\{{I_n}\}$ where $I_1=\overline{I}$ and $I_n= \overline{I^{n}}$ for $n\gg0$. Let $M, N$ be MCM  $A$-module then $\Ext_A^1(M,N)/T_{\mathcal{F}}(M,N)$ has finite length.	
\end{theorem}
\begin{proof}
We prove this theorem by induction. If $\dim A=1$, then  $\Ext^1_A(M,N)$ has finite length. In fact, for any prime ideal $\mathfrak{p}\ne \mathfrak{m}$, $(\Ext^1_A(M,N))_{\mathfrak{p}}=0$ because $A$ is reduced. Note that for dimension one case we do not need any assumption on residue field.\\
We now assume $\dim A \geq 2$ and result is true for dimension $d-1$.\\
Let $\alpha : 0\rt N\rt E\rt M\rt 0\in \Ext^1_A(M,N)$ and $a\in I$. Then we have following pushout diagram of $R$-modules for all $n\geq 1$
\begin{center}
	\begin{tikzcd}
	\alpha: 0 \arrow[r, "" ]  &N \arrow[r, ""] \arrow[d, "a^n" ] &E \arrow[r, ""]\arrow[d ] &M \arrow[r, ""] \arrow[d, "1_{M_R}" ] &0\\
	a^n\alpha: 0 \arrow[r, ""] &N \arrow[r ] &E_{n} \arrow[r, ""] &M \arrow[r, ""] &0
	\end{tikzcd}
\end{center}
Set $\mathcal{V} = \{ M, N, E, E_n \colon n \geq 1 \}$ and set  $$\mathcal{E} = \{ A\} \cup \{ \Syz^A_i(U) \colon  i = 0, 1 \ \text{and} \ U \in \mathcal{V} \}.$$
We now do the base change as described in \ref{BS}.II.(iii).
$$A\rt \hat{A}\rt B=\hat{A}[X_1,\ldots,X_t]_{\mathfrak{m}\hat{A}[X_1,\ldots,X_t]}$$
 For any MCM $A$-module $L$, set  $L_B = L\otimes_{A}B$.

Let $\mathcal{F}_B = \{ I_nB \}_{n \in \mathbb{Z}}$.
From  \ref{BS}.II.(iii), for all $n\geq 1$,
 $\zeta$ is $\mathcal{F}_B$-superficial for
 $$B \oplus M_B\oplus N_B \oplus E_{n,B}\oplus \Syz^B_1(M_B)\oplus \Syz^B_1(N_B)\oplus \Syz^B_1(E_{n,B}).$$
 Set $C = B/\zeta B$, $\mathcal{F}_C = \{ I_nC \}_{n \in \mathbb{Z}}$. Then $C$ is analytically unramified with $\dim C = d -1$. Furthermore $I_1C = \ov{I_1C}$ and $I_nC = \ov{I^nC}$ for $n \gg 0$.
 From \ref{good}, we have for all $n \geq 0$,
 $$  e^T_\mathcal{F}(a^n \alpha) = e^T_{\mathcal{F}_C,C}(a^n\alpha\otimes C)=e^T_{\mathcal{F}_C,C}(\overline{a^n}(\alpha\otimes C)).$$
  But from the assumption result is true for $C$. So
  $$e^T_{\mathcal{F}_C, C}(a^n\alpha\otimes C )=e^T_{\mathcal{F}_C, C}(\overline{a^n}(\alpha\otimes C))=0 \quad \text{for $n\gg 0$}. $$
   This implies $e^T_{\mathcal{F},A}(a^n\alpha)=0$ for $n\gg 0$.
    Let $I=(a_1,\ldots,a_u)$.  It follows that $$(a_1^{n_1},\ldots,a_u^{n_u})\Ext^1_A(M,N)\sub \T(M, N).$$
  So $\Ext_A^1(M,N)/T_{\mathcal{F}}(M,N)$ has finite length.
	\end{proof}

\begin{theorem}
	Let $(A,\mathfrak{m})$ be a Cohen-Macaulay analytically unramified  local ring of dimension $d$ with  residue field $k$. Suppose $k$ is  perfect field.  Let $I$ be an  $\mathfrak{m}$-primary ideal and $\mathcal{F}=\{{I_n}\}$ where $I_1=\overline{I}$ and $I_n= \overline{I^{n}}$ for $n\gg0$. Let $M, N$ be MCM  $A$-module then $\Ext_A^1(M,N)/T_{\mathcal{F}}(M,N)$ has finite length.	
\end{theorem}
\begin{proof}  By \ref{BS} we may assume $A$ is complete.
If $k$ is uncountable, result follows from Theorem \ref{main1}.
	
	Now we consider the case when  $k$ is finite or countably infinite. Then by \ref{BS}(iv) there exists a flat local extension $(B,\n)$ of $A$ with $\m B = \n$ which behaves well with respect to integral closure such that the residue field of  $B$ is uncountable.
		 Set $\mathcal{F}_B=\{{I_nB}\}_{n \in \mathbb{Z}}$, $M_B=M\otimes B$ and $N_B=N\otimes B$. Also note that ${I_n}B=\overline{I^nB}$ for  $n\gg 0$.
	
	Let $\alpha\in $ Ext$^1_A(M,N)$ and $a\in \mathfrak{m}$. Then for all $n\geq 1$ it is easy to see $$e^T_{\mathcal{F}}(a^n\alpha)=e^T_{\mathcal{F}_B}((a^n\alpha)\otimes B)=e^T_{\mathcal{F}_B}((a^n\otimes 1)(\alpha\otimes B)).$$
	From Theorem \ref{main1}, $e^T_{\mathcal{F}_B}((a^n\otimes 1)(\alpha\otimes B))=0$ for $n\gg 0$. So $e^T_{\mathcal{F}}(a^n\alpha)=0$ for $n\gg 0$. Therefore $a^n\alpha \in T_\mathcal{F}(M,N)$ for $n\gg 0$. Now the result follows from the similar argument as in Theorem \ref{main1}.
	\end{proof}
\section{Weak Brauer-Thrall-II }
We need the following two results.
\begin{lemma}\label{G_Fses}
	Let $(A,\mathfrak{m})$ be an \anram\ \CM\ local ring of dimension $d\geq 1$. Let $I$ be an $\mathfrak{m}$-primary ideal and $\mathcal{F}=\{{I_n}\}$ where $I_1=\overline{I}$ and $I_n= \overline{I^{n}}$ for $n\gg0$. If $M,N$ and $E$ are MCM modules and we have a $T_{\mathcal{F}}$-split sequence $0\rt N\rt E\rt M\rt 0$. Assume $G_{\mathcal{F}}(N)$ is \CM. Then we have short exact sequence $$0\rt G_{\mathcal{F}}(N)\rt G_{\mathcal{F}}(E)\rt G_{\mathcal{F}}(M)\rt 0.$$
	Furthermore, $e_i^{\mathcal{F}}(E)= e_i^{\mathcal{F}}(N)+e_i^{\mathcal{F}}(M)$ for $i=0,\ldots,d$.
\end{lemma}
\begin{proof}
	Follows from an argument similar to \cite[Lemma 6.3]{Pu trn}.
\end{proof}
\begin{proposition}\label{T_Ffinitelength}
	Let $(A,\mathfrak{m})$ be an \anram\ \CM\ local ring of dimension $d\geq 1$, $I$ an $\mathfrak{m}$-primary ideal, and $\mathcal{F}=\{\overline{I^n}\}_{n\geq0}$.
Assume the residue field $k = A/\m$  is either uncountable or a perfect field.
Let $M$ and $N$ be MCM $A$-modules with $G_{\mathcal{F}}(M)$ and $G_{\mathcal{F}}(N)$ \CM. If there exists only finitely many non-isomorphic MCM $A$-modules $E$ with $G_{\mathcal{F}}(E)$ \CM\ and $e^{\mathcal{F}}(E)= e^{\mathcal{F}}(N)+e^{\mathcal{F}}(M)$, then $T_{\mathcal{F}}(M,N)$ has finite length (in particular Ext$^1_A(M,N)$ has finite length).
\end{proposition}
\begin{proof}
	Follows from  an argument similar to \cite[Theorem 7.1]{Pu trn}.
\end{proof}
The following result is well-known. We indicate a proof for the convenience of the reader.
\begin{lemma}\label{intclosum}
	Let $(A,\mathfrak{m})$ be a Noetherian local ring and $I$ be an ideal of $A$. Set $B=A[X]$ and $J=(I,X)$ then $\overline{J^n}=\sum_{i=0}^{n} \overline{I^{n-i}}X^i$.
\end{lemma}
\begin{proof}
	Consider Rees algebra of $I$, $\mathscr{R}(I)= A[It] =  A\oplus It\oplus I^2t^2\oplus \ldots $.
	Its integral closure   in $A[t]$ is $\overline{\mathscr{R}(I)}= A\oplus \overline{I}t\oplus \overline{I^2}t^2\oplus \ldots $. By \cite[Chapter 5, Exercise 9]{AM} we get that
	$\overline{\mathscr{R}(I)}[X]$ is integral closure of ${\mathscr{R}(I)}[X]$ in $A[t][X]$. Comparing homogeneous components for all $n$ we get $\overline{(I,X)^n}=\sum_{i=0}^{n}\overline{I^{n-i}}X^i$.
\end{proof}
\begin{proposition}\label{d1_dimExt}
	Let $(A,\mathfrak{m})$ be an \anram\ Cohen-Macaulay local ring of dimension one and $I$ an $\mathfrak{m}$-primary ideal. Set $R=A[X]_{(\mathfrak{m},X)}$, $J=(I,X)$, $\mathcal{I}=\{\overline{I^n}\}_{n\geq0}$ and $\mathcal{J}=\{\overline{J^n}\}_{n\geq0}$. Then there exists an MCM $R$-module $E$ with $G_{\mathcal{J}}(E)$ Cohen-Macaulay and dim \normalfont{Ext}$^1_R(E,E)>0$.
\end{proposition}
\begin{proof}
	Let $M$ be an MCM $A$-module. Fix large enough $n$ (say $n_0$), then it is easy to see $N=\overline{I^{n_0}}M$ is MCM and $G_{\mathcal{I}}(N)$ is Cohen-Macaulay. From lemma \ref{intclosum}, we get $G_{\mathcal{J}}(N\otimes R)= G_{\mathcal{I}}(N)[X]$. So $G_{\mathcal{J}}(N\otimes R)$ is Cohen-Macaulay.
	
	From \cite[Theorem A.11(b)]{Bruns}, dim Ext$^1_R(N\otimes R, N\otimes R) >0$.
\end{proof}

\begin{theorem}\label{dimExt}
	Let $(A,\mathfrak{m})$ be an \anram\ \CM\ local ring of dimension $d\geq1$, and $I$ an $\mathfrak{m}$-primary ideal. Set $R=A[X_1,\ldots,X_m]_{(\mathfrak{m},X_1,\ldots,X_m)}$, $J=(I,X_1,\ldots,X_m)$, $\mathcal{I}=\{\overline{I^n}\}_{n\geq0}$ and $\mathcal{J}=\{\overline{J^n}\}_{n\geq0}$. Also set $S = \widehat{R}$ and $\mathcal{K}=\{\overline{J^n}S\}_{n\geq0}$. If $A$ has an MCM module $M$ with $G_{\mathcal{I}}(M)$ \CM. Then
\begin{enumerate}[\rm (1)]
  \item $M\otimes R$ is an MCM $R$-module with $G_{\mathcal{J}}(M\otimes R)$ \CM\ and \\ $\dim \Ext^1_R(M\otimes R,M\otimes R)>0$.
  \item  We  have $\ov{J^nS} = \ov{J^n}S$ for all $n \geq 1$. Furthermore $M\otimes S$ is an MCM $S$-module with $G_{\mathcal{K}}(M\otimes S)$ \CM\ and $\dim \Ext^1_S(M\otimes S,M\otimes S)>0$.
\end{enumerate}
\end{theorem}
\begin{proof}
(1)	It is sufficient to prove the result for $n=1$. So we can assume $R=A[X]$. It is easy to see $M\otimes R$ is MCM $R$-module and $G_{\mathcal{J}}(M\otimes R)=G_{\mathcal{I}}(M)[X]$ (follows from Lemma \ref{intclosum}). So $G_{\mathcal{J}}(M\otimes R)$ is \CM.
	
	From \cite[Theorem A.11(b)]{Bruns}, dim Ext$^1_R(M\otimes R, M\otimes R) >0$.

(2) The assertion $\ov{J^nS} = \ov{J^n}S$ for all $n \geq 1$ follows from \cite[9.1.1]{HS}. For the rest observe that
$M\otimes_A S = (M \otimes_A R)\otimes_R S$. This gives $\dim \Ext^1_S(M\otimes S,M\otimes S)>0$. Furthermore $G_\mathcal{K}(S)$ is a flat extension of $G_\mathcal{J}(R)$ with zero-dimensional fiber.  Notice
$$G_{\mathcal{K}}(M\otimes S) = G_{\mathcal{J}}(M\otimes R)\otimes_{G_\mathcal{J}(R)} G_\mathcal{K}(S).$$
By Theorem \cite[23.3]{Matsu} the result follows.
\end{proof}
\begin{theorem}
	Let $(A,\mathfrak{m})$ be a complete reduced \CM\ local ring of dimension $d\geq1$  and $I$ an $\mathfrak{m}$-primary ideal. Assume the residue field $k = A/\m$ is either  uncountable or perfect. Set $R=A[[X_1,\ldots,X_m]]$, $J=(I,X_1,\ldots,X_m)$, $\mathcal{I}=\{\overline{I^n}\}_{n\geq0}$, and $\mathcal{J}=\{\overline{J^n}\}_{n\geq0}$. If $A$ has an MCM module $M$ with $G_{\mathcal{I}}(M)$ \CM \ then  $R$ has infinitely many non-isomorphic MCM modules $D$ with $G_{\mathcal{J}}(D)$ \CM\ and bounded multiplicity.
\end{theorem}
\begin{proof}
	Follows from \ref{G_Fses}, \ref{T_Ffinitelength} and \ref{dimExt}.
\end{proof}
\section{Some results about $e^T_A()$}
In this section we prove Theorem \ref{proj-intro} (see Theorem \ref{proj}).
\begin{lemma}\label{subadd}
	Let $(A,\mathfrak{m}, k)$ be a \CM \ local ring and $M,N$ be MCM $A$-modules. Let $\alpha$ be $T$-split and $\alpha'$ be any extension, then $e^T_A(\alpha+\alpha')\leq e^T_A(\alpha')$. Also, if char$(A)=p^n>0$, then $e^T(\alpha+\alpha')=e^T_A(\alpha')$.
\end{lemma}
\begin{proof}
	Let $\alpha$ can be represented as $0\rt N\rt E\rt M\rt 0$ and $\alpha'$ as $0\rt N\rt E'\rt M\rt 0$.
	Consider following pullback diagram
		\begin{center}
				\begin{tikzcd}
		\beta: 0 \arrow[r, "" ]  &N \arrow[r, ""] \arrow[d, "1_N" ] &E'' \arrow[r, ""]\arrow[d ] &E' \arrow[r, ""] \arrow[d, "" ] &0\\
		\alpha: 0 \arrow[r, ""] &N \arrow[r ] &E \arrow[r, ""] &M \arrow[r, ""] &0
		\end{tikzcd}
	\end{center}
From \ref{usplit}, $\beta $ is $T$-split. So, 	$\e(E'')=\e(N)+\e(E')$.

Now $\alpha+\alpha'$ can be written as $0\rt N\rt Y\rt M\rt 0$ where $Y=E''/S$ and $S=\{(-n,n)\in E''| n\in N\}$. So we have following commutative diagram
 \begin{center}
 	\begin{tikzcd}
 	\beta: 0 \arrow[r, "" ]  &N \arrow[r, ""] \arrow[d, "1_N" ] &E'' \arrow[r, ""]\arrow[d, "\delta" ] &E' \arrow[r, ""] \arrow[d, "" ] &0\\
 	\alpha+\alpha': 0 \arrow[r, ""] &N \arrow[r ] &Y \arrow[r, ""] &M \arrow[r, ""] &0
 	\end{tikzcd}
  \end{center}
	Here $\delta $ is natural surjection.
	
	Now from the exact sequence $\gamma : 0\rt N\rt E''\xrightarrow{\delta} Y\rt 0$ we get $\e(\gamma)=\e(N)+\e(Y)-\e(E'')$.\\
	Now we get
	\begin{align*}
	e^T_A(\alpha+\alpha')&=e^T_A(N)+\e(M)-\e(Y)\\
	&=\e(\alpha')+\e(E') -\e(Y)\\
	&= \e(\alpha')+\e(E'')-\e(N)-\e(Y)\\
	&= \e(\alpha') -\e(\gamma)
	\end{align*}
	So, $e^T_A(\alpha+\alpha')\leq \e(\alpha').$
	
	If char$(A)=p^n >0$, then we have $$\e(\alpha')\leq \e((p^n-1)\alpha+\alpha')\leq \ldots\leq e^T_A(\alpha+\alpha')\leq \e(\alpha') .$$
	Note that $p^n\alpha=0$ is split exact sequence.\\
	This implies $e^T_A(\alpha+\alpha')= \e(\alpha').$
		\end{proof}

Let $\mathbb{N}$ be the set of non-negative integers.
\begin{remark}\label{factor}
  If $char(A) = p^n > 0$ then we have a well defined function \\ $[e^T_A] \colon \Ext^1_A(M, N)/T_A(M, N) \rt \mathbb{N}$.
\end{remark}
If $V$ is a vector-space over a field $k$ then let $\mathbb{P}(V)$ denote the projective space determined by $V$.
\begin{theorem}\label{proj}
(with hypotheses as in \ref{subadd}) Further assume $A$ is of characteristic $p > 0$ and that $A$ contains a field $k \cong A/\m$. If $\Ext_A^1(M, N) \neq T_A(M,N)$ then the function $[e^T_A]$ defined in \ref{factor} factors as
$$ [\ov{e^T_A}] \colon \mathbb{P}(\Ext^1_A(M, N)/T_A(M, N)) \rt \mathbb{N}\setminus{0}. $$
\end{theorem}
\begin{proof}
  	Let $\alpha \in \Ext^1_A(M,N)$  be represented as $0\rt N\rt E\rt M\rt 0$. Let $r \in k^*$ and let  $r\alpha$  be represented as $0\rt N\rt E'\rt M\rt 0$.
	Consider the diagram
		\begin{center}
				\begin{tikzcd}
		\alpha: 0 \arrow[r, "" ]  &N \arrow[r, ""] \arrow[d, "r" ] &E \arrow[r, " "]\arrow[d, "\psi" ] &M \arrow[r, ""] \arrow[d, "1_M" ] &0\\
		r\alpha: 0 \arrow[r, ""] &N \arrow[r ] &E' \arrow[r, ""] &M \arrow[r, ""] &0
		\end{tikzcd}
	\end{center}
Note $\psi \colon E \rt E'$ is an isomorphism. It follows that $e^T(\alpha) = e^T(r\alpha)$. The result follows.
\end{proof}
\begin{remark}
All the results in this section is also true for $\eT()$. Same proof works in that case also.
\end{remark}

\s \label{spread} For the rest of this section we consider the following setup:\\
$(A,\m)$ is a complete reduced CM local ring. Also assume $A$ contains a field $k \cong A/\m$. Furthermore $k$ is either uncountable or a perfect field. Let $I$ be an $\m$-primary ideal and let
$\mathcal{F} = \{ I_n \}_{n \in \mathbb{Z}}$ be an $I$-admissible filtration with $I_1 = \ov{I}$ and $I_n = \ov{I^n}$ for $n \gg 0$. Let $M, N$ be MCM $A$-modules and consider the function
\begin{align*}
  e^T_\mathcal{F} &\colon \Ext_A^1(M, N) \rt \mathbb{N} \\
  \alpha &\mapsto e^T(\alpha).
\end{align*}
Notice $e^T_\mathcal{F}(\alpha) \leq e^T_\mathcal{F}(M) + e^T_\mathcal{F}(N)$. So  $e^T_\mathcal{F}(\Ext_A^1(M,N))$ is a bounded set.

If $Z$ is a finite set then set $|Z|$ denote its cardinality. \\ Set $Z_\mathcal{F}(M, N) = | e^T_\mathcal{F}(\Ext_A^1(M,N))|$.

\begin{corollary}(with hypotheses as in \ref{spread}). Further assume $k$ is a finite field and $\Ext^1_A(M, N)$ is non-zero and has finite length as an $A$-module (and so a finite dimensional $k$-vector space).
Set $c(M, N) = | \mathbb{P}(\Ext_A^1(M, N))|$. Let $I$ be any $\m$-primary ideal and let $\mathcal{F} = \{ I_n \}_{n \in \mathbb{Z}}$ be an $I$-admissible filtration with $I_1 = \ov{I}$ and $I_n = \ov{I^n}$ for $n \gg 0$. Then $Z_\mathcal{F}(M, N) \leq c(M, N)$.
\end{corollary}
\begin{proof}
  We may assume $T_\mathcal{F}(M, N) \neq \Ext^1_A(M, N)$. By \ref{proj} we get that
  $$ Z_\mathcal{F}(M, N) \leq  | \mathbb{P}(\Ext^1_A(M, N)/T_\mathcal{F}(M, N)) |. $$
  Note $| \mathbb{P}(\Ext^1_A(M, N)/T_\mathcal{F}(M, N)) |$ is bounded above by $c(M,N)$. The result follows.
\end{proof}
\section{T-split sequences and AR-sequences}
The goal of this section is to prove the following result:
\begin{theorem}
  \label{AR-Tsplit}
  Let $(A,\m)$ be a Henselian \CM \  local ring and let $M$ be an indecomposable MCM $A$-module free on the punctured spectrum of $M$. The following assertions are equivalent:
  \begin{enumerate}[\rm (i)]
    \item  There exists a $T$-split sequence $\alpha \colon 0 \rt K \rt E \rt M \rt 0$ with $\alpha$ non-split.
    \item  There exists a $T$-split sequence $\beta \colon 0 \rt V \rt U \rt M \rt 0$  with $V$ indecomposable and $\beta$ non-split.
    \item The AR-sequence ending at $M$ is $T$-split.
  \end{enumerate}
\end{theorem}
For definition of Auslander-Reiten (AR) sequences, see \cite[Chapter 2]{Yoshino}.
	From \cite[Theorem 3.4]{Yoshino}, we know that for an indecomposable MCM module over $A$, then there is an AR sequence ending in $M$ if and only if $M_{\mathfrak{p}}$ is free for all $\mathfrak{p}\in $ Spec$^0(A)=$ Spec$(A)\setminus \{\mathfrak{m}\}$

Before proving Theorem \ref{AR-Tsplit}   we need the following well-known result.
We give a proof for the convenience of the reader.
\begin{lemma}\label{nonsplit}
	Let $A$ be a Noetherian ring and $N,M$ and $E$ are finite $A$-module. Let $N=N_1\oplus N_2$ and we have following diagram
		\begin{center}
		\begin{tikzcd}
		s: 0 \arrow[r, "" ]  &N \arrow[r, "f"] \arrow[d, "p_i" ] &M \arrow[r, ""]\arrow[d,"\gamma_i" ] &E \arrow[r, ""] \arrow[d, "1_E" ] &0\\
		s_i: 0 \arrow[r, ""] &N_i \arrow[r,"f_i" ] &M_i \arrow[r, ""] &E \arrow[r, ""] &0
		\end{tikzcd}
	\end{center}

	for $i=1,2$. Here $p_i: N\rt N_i$ is projection map for $i=1,2$. If $s$ is non-split then one of the $s_i$ is non-split.
\end{lemma}
\begin{proof}
	Let $s_1$ and $s_2$ are split exact sequences. So we have $g_i: M_i\rightarrow N_i$ for $i=1,2$ such that $g_if_i=1_{N_i}$. Consider function $g= (g_1\gamma_1, g_2\gamma_2): M\rightarrow N_1\oplus N_2$. Let $(n_1,n_2)\in N$ then
	\begin{align*}
	gf(n_1,n_2)&=(g_1\gamma_1f(n_1,n_2), g_2\gamma_2f(n_1,n_2))\\
	&=(g_1f_1p_1(n_1,n_2),g_2f_2p_2(n_1,n_2))\\
	&=(n_1,n_2)
	\end{align*}
 This implies $g$ is a  left inverse of $f$, so $s$ is split exact sequence.
 	\end{proof}
 We now give
 \begin{proof}[Proof of Theorem \ref{AR-Tsplit}]
 The assertions (iii) $\implies$ (ii) $\implies$ (i)  are clear.

 (i) $\implies$ (ii). As $A$ is Henselian the module $K$ splits as a sum of indecomposable modules $K = K_1\oplus K_2 \oplus \cdots \oplus K_r$. The result follows from Lemma \ref{nonsplit}.

 (ii) $\implies$ (iii).
 Let $\beta : 0\rt V \rt U\rt M\rt 0$ be  $T$-split and $\beta$ non-split. As $V$ is indecomposable, we have following diagram
		\begin{center}
		\begin{tikzcd}
		 \beta :0 \arrow[r, "" ]  &V\arrow[r, ""] \arrow[d, "" ] & U \arrow[r, ""]\arrow[d ] &M \arrow[r, ""] \arrow[d, "1_M" ] &0\\
		s: 0 \arrow[r, ""] &\tau(M) \arrow[r ] &E \arrow[r, ""] &M \arrow[r, ""] &0
		\end{tikzcd}
	\end{center}
	Here $s$ is AR-sequence ending in $M$, see \cite[2.8]{Yoshino}. This implies $s$ is $T$-split (see \cite[Proposition 3.8]{Pu trn}).
	
 \end{proof}
\section{Some observation about complete intersection}
In this section we prove Theorem \ref{cx2} (see Theorem \ref{cx2-body}).

 \s Let $(A,\mathfrak{m})$ be Noetherian local ring and $M$ a finite $A$-module. We denote the $n$-th Betti number of $A$ module $M$ as $\beta^A_n(M)$. Then complexity of $M$ can be defined as
\[
\text{cx}~ M = \text{inf} \left\{ r\in \mathbb{N}\middle\vert \begin{array}{l}
	\text{there exists polynomial}\ p(t)\ \text{of degree}\ r-1\\
	\text{such that}\ \beta^A_n(M)\leq p(n)\ \text{for}\ n\gg 0
	\end{array}\right\}
\]

\s Let $Q$ be a Noetherian ring  and $f=f_1,\ldots,f_c$ be a $Q$-regular sequence. Set $A=Q/(f_1,\ldots,f_c)$. Let $M$ be a finite $A$-module. Note projdim$_QM<\infty$.

Let $\mathbb{F}$ is a free resolution of $M$ as $A$-module. Let $t_1,\ldots,t_c:\mathbb{F}(+2)\rt \mathbb{F}$ be the Eisenbud operators (see \cite[Section 1]{Eisenbud}). Consider the polynomial ring $B=A[t_1,\ldots,t_c]$ with deg$(t_i)=2$ for $i=1,\ldots,c.$ Let $L$ be an $A$-module, then we can think of  Tor$_*^A(M,L)=\bigoplus_{i\geq 0}\text{Tor}^A_i(M,L)  $ as $B$-module (here we give degree $-i$ for an element of Tor$_i^A(M,L)$).

\begin{theorem}\label{cx2-body}
	Let $(Q,\mathfrak{n},k)$ be a Henselian regular local ring and $\underline{f} = f_1,\ldots,f_c\in \mathfrak{n}^2$ a regular sequence. Assume $k$ is infinite. Set $I=(f_1,\ldots,f_c)$ and $(A,\mathfrak{m})= (Q/I,\mathfrak{n}/I)$.  Assume $\dim A = 1$. Let $M$ be an indecomposable MCM $A$-module with $\cx_AM\ge  2$ and
	$$\mathbb{F}: \ldots \rt F_{n+1}\rt F_n\rt F_{n-1}\rt \ldots$$
	be the minimal free resolution of $M$. Set $M_r= \Syz_r^A(M)$. Then there exists $r_0$ such that for all  $r\geq r_0$, there are exact sequences  $ \alpha_r \colon 0\rt K_r\rt M_{r+2}\rt M_r\rt 0$
such that \begin{enumerate}[\rm (1)]
            \item $\cx K_r \leq \cx M - 1$ for $r \geq r_0$.
            \item $\alpha_r$ is non-split for   $r \geq r_0$.
            \item  $\alpha_r$ is $T$-split for   $r \geq r_0$
          \end{enumerate}
          If furthermore  $M$ is free on the punctured spectrum of $A$ then the AR-sequence ending at $M_r$ is $T$-split for all $r \geq r_0$.
\end{theorem}
\begin{proof}
Let $x$ be an $A$-superficial element. The map $\alpha_n \colon A/{\mathfrak{m}^{n}}\rt A/{\mathfrak{m}^{n+1}}$ defined by $\alpha(a+\mathfrak{m}^n) = ax+\mathfrak{m}^{n+1}$ induces an isomorphism of $\Tor^A_i(A/\mathfrak{m}^n,M)$ and $\Tor^A_i(A/\mathfrak{m}^{n+1},M)$  for $n \geq \redA(A)$(see \cite[Lemma 4.1(3)]{Pu_GrowthSyz}).
	
	 Fix $n_0\geq \redA(A)$. For $j = 1,\ldots, n_0$ we have\\  $\Tor_*^A(A/\mathfrak{m}^{j},M)= \bigoplus_{i\geq 0}\Tor_i^A(A/\mathfrak{m}^{j},M)$ is *-Artinian $B=A[t_1,\ldots,t_c]$ module, where $t_1,\ldots,t_c$ are Eisenbud operators. Then for $i\gg 0$ (say $i \geq i_0$) and for $j = 1, \ldots, n_0$ we have following exact sequence $$\Tor_{i+2}^A(A/\mathfrak{m}^{j},M)\xrightarrow{\xi} \Tor_{i}^A(A/\mathfrak{m}^{j},M) \rt 0.$$ Here $\xi$ is a linear combination of $t_1, \ldots, t_c$. (see \cite[Lemma 3.3]{Eisenbud}).

We have following commutative diagram for $i\geq i_0$
	 \begin{center}
	 	\begin{tikzcd}
	 	 \text{Tor}_{i+2}^A(A/\mathfrak{m}^{n_0},M) \arrow[r, "\xi"]\arrow[d,"\theta_{i+2}" ] &\text{Tor}_{i}^A(A/\mathfrak{m}^{n_0},M) \arrow[r, ""] \arrow[d, "\theta_i"] &0\\
	 	 \text{Tor}_{i+2}^A(A/\mathfrak{m}^{n_0+1},M) \arrow[r, "\xi"] &\text{Tor}_{i}^A(A/\mathfrak{m}^{n_0+1},M)  & \
	 	\end{tikzcd}
	 \end{center}
	where $\theta_i = \Tor^A_i(\alpha_{n_0}, M)$. As $\theta_i$ and $\theta_{i+2}$ are isomorphism we get that the bottom row is also surjective. Iterating we get an exact sequence for all $j \geq 1$ and for all $i \geq i_0$,
 \begin{equation*}\label{dagL}
\Tor_{i+2}^A(A/\mathfrak{m}^{j},M)\xrightarrow{\xi} \Tor_{i}^A(A/\mathfrak{m}^{j},M) \rt 0.
\end{equation*}
	
Note $\xi$ induces a chain map $\xi \colon \mathbb{F}[2] \rt \mathbb{F}$. As we have a surjection
$$\Tor_{i+2}^A(A/\mathfrak{m},M)\xrightarrow{\xi} \Tor_{i}^A(A/\mathfrak{m},M) \rt 0, \quad \text{for $i \geq i_0$,} $$
by Nakayama Lemma we have surjections $F_{i+2} \xrightarrow{\xi} F_i$ for all $i \geq i_0$ (say with kernel $G_i$).
Notice we have a short exact sequence of complexes $$0 \rt \mathbb{G}_{\geq i_0} \rt \mathbb{F}[2]_{\geq i_0} \xrightarrow{\xi} \mathbb{F}_{\geq i_0}\ \rt 0.$$
Thus we have surjections $M_{i+2}  \xrightarrow{\xi} M_i$ for all $i \geq i_0$, say with kernel $K_i$.
We note that $	\mathbb{G}_{\geq i_0}$ is a free resolution of $K_{i_0}$ and that $K_i$ is (possibly upto a free summand) the $(i-i_0)^{th}$ syzygy of $K_{i_0}$. It follows that $\cx K_{i} = \cx K_{i_0} \leq \cx M -1$. We have an exact sequence $\alpha_r \colon 0 \rt K_r \rt M_{r+2} \rt M_r \rt 0$ for all $r \geq r_0$.
Since $M$ is indecomposable, $M_r = \Syz_r^A(M)$ is also indecomposable for all $r \geq 1$  (see \cite[Lemma 8.17]{Yoshino}). As $\cx M \geq 2$ it follows that $M_{r+2} \ncong M_r$ for all $r \geq 1$. It follows that $\alpha_r$ is not split for all $r \geq r_0$.

By \ref{dagL} it follows that for $i \geq i_0$ we have an an exact seqquence
$$ 0 \rt \Tor^A_1(A/\m^j, K_i) \rt \Tor^A_1(A/\m^j, M_{i+2}) \rt \Tor^A_1(A/\m^j, M_i) \rt 0,$$
for all $j \geq 1$. Clearly this implies that $\alpha_i$ is $T$-split.

Notice $M_r$ is free on Spec$^0(A)$ for all $r \geq 1$. As $\alpha_r$ is $T$-split, it follows from \ref{AR-Tsplit} that the AR-sequence ending at $M_r$ is $T$-split for all $r \geq r_0$.
	\end{proof}
	
\section{$T$-split sequences on hypersurfaces defined by quadrics}
In this section we prove Theorem \ref{quadric} (see Theorem \ref{quadric-body}). We also construct example \ref{counter} (see \ref{ex-non-split}).
\s \label{setup-q} In this section  $(Q,\mathfrak{n})$ is a Henselian regular local ring with algebraically closed residue field $k = Q/\mathfrak{n}$  and let $f \in \mathfrak{n}^2 \setminus \mathfrak{n}^3$. Assume the hypersurface $A  = Q/(f)$ is an isolated singularity of dimension $d \geq 1$.

\begin{remark}\label{rem-q}
  It is well-known that as $f$ is a quadric, the ring $A$ has minimal multiplicity.  It follows that $e_0(A) = 2$ and $e_1(A) = 1$. We also have  that if $M$ is MCM then $N = \Syz^A_1(M)$  is Ulrich, i.e., $\mu(N) = e_0(N)$ (furthermore $e_1(N) =0)$. As $A$ is also Gorenstein we get that any MCM $A$-module $M \cong F \oplus E$ where $F$ is free and $E$ has no-free summands and is a syzygy of an MCM $A$-module; in particular $E$ is Ulrich.
\end{remark}
The following results computes $e^T(-)$ for MCM $A$-modules. We also give a sufficient condition for a short exact sequence to be $T$-split.
\begin{proposition}\label{compute-et}
(with hypotheses as in \ref{setup-q}) Let $M, N, U, V$ be MCM $A$-modules with $M, N$ having no free-summands. Then
\begin{enumerate}[\rm (1)]
  \item $e^T(M) = \mu(M)$.
  \item Let $U = L \oplus F$ where $F$ is free and $L$ has no free summands. Then $e^T(U) = \mu(L)$.
  \item Let $\alpha \colon 0 \rt N \rt V \rt M \rt 0$. If $\mu(V) = \mu(N) + \mu(M)$ then
  \begin{enumerate}[\rm (a)]
    \item $V$ is Ulrich
    \item  $\alpha$ is $T$-split.
  \end{enumerate}
\end{enumerate}
\end{proposition}
\begin{proof}
(1) Note $\Syz^A_1(M)$ is also Ulrich.  Using \ref{rem-q}  we have
$$ e^T(M) = e_1(A)\mu(M) - e_1(M) - e_1(\Syz^A_1(M)) = \mu(M). $$

(2) Note $e^T(U) = e^T(L) + e^T(F) = \mu(L) + 0 = \mu(L)$.

(3) We have
\begin{align*}
  e_0(V) &= e_0(M) + e_0(N), \\
   &=  \mu(M) + \mu(N) \quad \text{as $M, N$ are Ulrich,} \\
   &= \mu(V).
\end{align*}
In particular $V$ is Ulrich. So $V$ has no free summands. We have
$$ e^T(\alpha) = e^T(M) + e^T(N) - e^T(V) = \mu(M) + \mu(N) - \mu(V) = 0. $$
So $\alpha$ is $T$-split.
\end{proof}
We now state and prove the main result of this section.
\begin{theorem}
\label{quadric-body}(with hypotheses as in \ref{setup-q}) All but a finitely many AR-sequences of $A$ are $T$-split.
\end{theorem}
\begin{proof}
We may assume that $A$ is of infinite CM \ representation type (i.e., there exists infinitely many mutually non-isomorphic   indecomposable MCM $A$-modules) otherwise there is nothing to prove.
The AR-quiver of $A$ is locally finite graph, \cite[5.9]{Yoshino}. It follows that for all but finitely many MCM indecomposable $A$-modules the middle term of the AR-sequence ending at $M$ and $\Syz^A_1(M)$ will not contain a free summand. Let $M$ be such a  indecomposable MCM $A$-module and let $s \colon  0 \rt \tau(M) \rt V \rt M \rt 0$ be the AR-sequence ending at $M$. Then by
\cite[7.11]{Pu-Ar} we have $\mu(V) = \mu(M) + \mu(\tau(M))$. By \ref{compute-et}(3) it follows that $s$ is $T$-split.
\end{proof}
We now give example of an AR-sequence which is not split.
\begin{example}
\label{ex-non-split}(with hypotheses as in \ref{setup-q}) Let $s \colon 0 \rt N \rt E \rt M \rt 0$ be an AR-sequence such that $E$ has a free summand. Then
\begin{enumerate}[\rm (1)]
  \item $s$ is NOT $T$-split.
  \item If $t \colon 0 \rt V \rt U \rt M \rt 0$ is any non-split exact sequence of MCM $A$-modules then $t$ is NOT $T$-split.
\end{enumerate}
\end{example}
\begin{proof}
(1) Note $\mu(N) \geq \mu(E) - \mu(M)$. Furthermore equality cannot hold for otherwise by Proposition \ref{compute-et} we will get $E$ is Ulrich, a contradiction.
Let $E = L \oplus F$ with $F \neq 0$ free and $L$ has no free summands.
We note that
\begin{align*}
  e^T(s) &= e^T(N) + e^T(M) - e^T(E), \\
   &= \mu(N) + \mu(M) - \mu(L) \\
   &> \mu(N) + \mu(M) - \mu(E)  > 0.\\
\end{align*}
Thus $s$ is NOT $T$-split.

(2) This follows from Theorem \ref{AR-Tsplit}.
\end{proof}
	\section{An Application of $T$-split sequences in Gorenstein case}	
In this section we prove Proposition \ref{relation} (see \ref{rel-body}). We also prove Theorem \ref{krull-schmidt} (see \ref{isom_th} and \ref{ks-body}).
	\s Let $(A,\mathfrak{m})$ be a Gorenstein local ring. Let CM$(A)$ denotes the category of MCM $A$-modules and \underline{CM}$(A)$ the stable category of CM$(A)$. Note that objects of \underline{CM}$(A)$ are same as the objects of CM$(A)$ and if $M$ and $N$ are MCM $A$-modules then $$\underline{\text{Hom}}_{A}(M,N)=\frac{\text{{Hom}}_A(M,N)}{\{f:M\rt N |~ f~~ \text{factors through a projective module}\}}.$$
	
\s\label{cosyz} (Co-syzygy) Let $(A,\mathfrak{m})$ be a Gorenstein local ring and $M$ be an MCM $A$-module. Let $M^*=$ Hom$(M,A)$, then $M^{**}\cong M$. Suppose $G\xrightarrow{\epsilon} F\rt M^*\rt 0$ is a minimal presentation of $M^*$. Dualizing  this, we get $0\rt M\rt F^*\xrightarrow{\epsilon^*} G^*$. Co-syzygy of $M$ can be defined as  coker$(\epsilon^*)$ and denoted as $\Omega^{-1}(M)$. So we have exact sequence $0\rt M\rt F\rt \Omega^{-1}(M)\rt 0 $.\\
 Note that co-syzygy does not depend on the minimal presentation, that is if we take another minimal presentation $G'\xrightarrow{\epsilon'} F'\rt M^*\rt 0$ then coker$(\epsilon^*)\cong$ coker$((\epsilon')^*)$.

\s	Let $\Omega^{-1}(M)$ be the co-syzygy of $M$, then we have following  exact sequence $$0\rt M\rt F\rt \Omega^{-1}(M)\rt 0 $$
	here $F$ is a free $A$-module (see \ref{cosyz}).\\
	For any $f\in \text{{Hom}}_A(M,N)$, we have following diagram
		\begin{center}
		\begin{tikzcd}
		 0 \arrow[r, "" ]  &M \arrow[r, ""] \arrow[d, "f" ] &F \arrow[r, ""]\arrow[d,"" ] &\Omega^{-1}(M) \arrow[r, ""] \arrow[d, "1" ] &0\\
		\alpha_f: 0 \arrow[r, ""] &N \arrow[r,"i" ] &C(f) \arrow[r, "p"] &\Omega^{-1}(M) \arrow[r, ""] &0
		\end{tikzcd}
	\end{center}
Here the first sequare is a pushout diagram.
\begin{remark}
	Note that \underline{CM}$(A)$ is an triangulated category with the projection of the sequence  $M\xrightarrow{f}N\xrightarrow{i} C(f)\xrightarrow{-p} \Omega^{-1}(M)$ in \underline{CM}$(A)$ as a basic triangles for any morphism $f$. Exact triangles are triangles isomorphic to a basic triangle (see \cite[4.7]{Buchw}). Also note that for any short exact sequence $0\rt U\rt V\rt W\rt 0$ in {CM}$(A)$, we have exact triangle $U\rt V\rt W\rt \Omega^{-1}(U)$ (see \cite[Remark 3.3]{Pu_stable}).
\end{remark}
\s Let $M$ and $N$ be  MCM $A$-modules, the
	it is easy to show that $$\text{\underline{Hom}}_A(M,N)\overset{\eta}\cong \text{Ext}^1_A(\Omega^{-1}(M),N) \ \text{as}\ A\text{-modules}.$$ In fact, the map $\eta : f\mapsto \alpha_f$ is an isomorphism. It is clear that $\eta$ is natural in $M$ and $N$.
	
	 Let $T_A(\Omega^{-1}(M),N)$ denotes the set of all $T$-split sequences in Ext$^1_A(\Omega^{-1}(M),N)$. If we denote $\eta^{-1}(T_A(\Omega^{-1}(M),N))$ by $\mathcal{R}(M,N)$. Then $\eta $ induces following isomorphism
	$$\frac{\text{\underline{Hom}}_A(M,N)}{\mathcal{R}(M,N)}\cong \frac{\text{Ext}^1_A(\Omega^{-1}(M),N)}{T_A(\Omega^{-1}(M),N)}.$$

\begin{proposition}\label{rel-body}
	$\mathcal{R}$ is a relation on $\CMS(A)$.
\end{proposition}
\begin{proof}
	To prove that $\mathcal{R}$ is a relation on \underline{CM}$(A)$ we  need to show :\\
	if $M_1,M,N,N_1\in $ CM$(A)$, $u\in \mathcal{R}(M,N)$, $f\in\underline{\Hom}_A(M_1,N)$  and $g\in\underline{\Hom}_A(N,N_1)$  then $u\circ f\in \mathcal{R}(M_1,N) $ and $g\circ u\in \mathcal{R}(M,N_1)$.
	
	We first prove $u\circ f\in \mathcal{R}(M_1,N) $. We have following diagram of exact traingles
	\begin{center}
		\begin{tikzcd}
		M_1 \arrow[r, "u\circ f" ] \arrow[d, "f" ] &N \arrow[r, ""] \arrow[d, "1" ] &C(u\circ f) \arrow[r, ""]\arrow[d,"h" ] &\Omega^{-1}(M_1)  \arrow[d, "" ] \\
		 M \arrow[r, "u"] &N \arrow[r,"" ] &C(u) \arrow[r, ""] &\Omega^{-1}(M)
		\end{tikzcd}
	\end{center}
Note that the map $h$ exists from the property (TR3) (see\cite[Definition 10.2.1]{Weibel}).

So we have following diagram of exact sequences
\begin{center}
	\begin{tikzcd}
	\alpha_{u\circ f}:0 \arrow[r, "" ]  &N \arrow[r, ""] \arrow[d, "1" ] &C(u\circ f)\oplus F \arrow[r, ""]\arrow[d,"h" ] &\Omega^{-1}(M_1)  \arrow[d, "" ]\arrow[r,""]&0 \\
	\alpha_u: 0 \arrow[r, ""] &N \arrow[r,"" ] &C(u)\oplus G \arrow[r, ""] &\Omega^{-1}(M) \arrow[r,""]&0
	\end{tikzcd}
\end{center}
where $F$ and $G$ are free $A$-modules.
Now since $u\in \mathcal{R}(M,N)$, this implies $\alpha_u$ is $T$-split. So from \cite[Proposition 3.9]{Pu trn}, $\alpha_{u\circ f}$ is $T$-split. In other words $u\circ f \in \mathcal{R}(M_1,N)$.

Next we prove $g\circ u\in \mathcal{R}(M,N_1)$.  We have following diagram of exact traingles
\begin{center}
	\begin{tikzcd}
	M \arrow[r, "u" ] \arrow[d, "1" ] &N \arrow[r, ""] \arrow[d, "g" ] &C(u) \arrow[r, ""]\arrow[d,"\theta" ] &\Omega^{-1}(M)  \arrow[d, "" ] \\
	M \arrow[r, "g\circ u"] &N_1 \arrow[r,"" ] &C(g\circ u) \arrow[r, ""] &\Omega^{-1}(M)
	\end{tikzcd}
\end{center}
  Note that the property (TR3) (see \cite[Definition 10.2.1]{Weibel}) guarantees the  existence of map $\theta$.

  So we have following diagram of exact sequences
  \begin{center}
  	\begin{tikzcd}
  	\alpha_{u}:0 \arrow[r, "" ]  &N \arrow[r, ""] \arrow[d, "g" ] &C(u)\oplus F' \arrow[r, ""]\arrow[d,"\theta" ] &\Omega^{-1}(M)  \arrow[d, "1" ]\arrow[r,""]&0 \\
  	\alpha_{g\circ u}: 0 \arrow[r, ""] &N_1 \arrow[r,"" ] &C(g\circ u)\oplus G' \arrow[r, ""] &\Omega^{-1}(M) \arrow[r,""]&0
  	\end{tikzcd}
  \end{center}
where $F'$ and $G'$ are free $A$-modules.
Now since $u\in \mathcal{R}(M,N)$, this implies $\alpha_u$ is $T$-split. So from \cite[Proposition 3.8]{Pu trn}, $\alpha_{g\circ u}$ is $T$-split. In other words $g\circ u \in \mathcal{R}(M_1,N)$.
	\end{proof}
\s Since $\mathcal{R}$ is a relation on \underline{CM}$(A)$, the factor category $\mathcal{D}_A=$ \underline{CM}$(A)/\mathcal{R}$ is an additive category. Note that objects of $\mathcal{D}_A$ are the same as those  of \underline{CM}$(A)$, and for any $M,N\in $ Obj$(\mathcal{D}_A)$, Hom$_{\mathcal{D}_A}(M,N)=\text{\underline{Hom}}_A(M,N)/\mathcal{R}(M,N)$.\\
Also note that $\ell(\text{Hom}_{\mathcal{D}_A}(M,N))<\infty$ (see \cite[Theorem 4.1]{Pu trn}).

Next we want to prove the main result of this section. But first we prove a lemma.

\begin{lemma}\label{Jac}
	Let $(A,\mathfrak{m})$ be a Henselian Gorenstein local ring and $M$ be an MCM $A$-module. Then $\mathcal{R}(M,M)\sub  $ \normalfont{Jac(\underline{End}}$_A(M))$ in \normalfont{\underline{CM}}$(A)$.
\end{lemma}
\begin{proof}
We prove this result in three cases:\\
{\bf Case 1:}
	 $M$  is  indecomposable MCM module.\\
	Let $u\in \mathcal{R}(M,M)$ and if possible assume that $u\notin$ Jac(\underline{End}$_A(M))$. This implies $u$ is invertible. Now we have following diagram of exact sequences
	\begin{center}
		\begin{tikzcd}
		0 \arrow[r, "" ]  &M \arrow[r, ""] \arrow[d, "u" ] &F \arrow[r, ""]\arrow[d,"" ] &\Omega^{-1}(M) \arrow[r, ""] \arrow[d, "1" ] &0\\
		\alpha_u: 0 \arrow[r, ""] &M \arrow[r,"" ] &C(u) \arrow[r, ""] &\Omega^{-1}(M) \arrow[r, ""] &0
		\end{tikzcd}
	\end{center}
From here we get $C(u)\cong F$. Also from the assumption $\alpha_u$ is $T$-split. We know that $e^T(\alpha_u)=e^T(M)+ e^T(\Omega^{-1}(M))-e^T(C(u))$.

 So $e^T(\alpha_u)=e^T(M)+ e^T(\Omega^{-1}(M))$ because $C(u)\cong F$. This implies \\  $e^T(\alpha_u)>0$ but this is a contradiction because $\alpha_u$ is $T$-split. Therefore, $u\in$ Jac(\underline{End}$_A(M))$.\\
  {\bf Case 2:} $M\cong E^n$ for some indecomposable MCM module $E$. \\
 It is clear that $\mathcal{R}(M,M)= \mathcal{R}(E^n,E^n)\cong M_n(\mathcal{R}(E,E))$. Here $M_n()$ denotes $n\times n$-matrix.\\
 We also know that \underline{End}$(E^n)\cong M_n(\text{\underline{End}}(E))$ and  \\
 Jac(\underline{End}$(E^n))\cong M_n(\text{Jac(\underline{End}}(E)))$.
 From the case (1), \\ $M_n(\mathcal{R}(E,E))\sub M_n(\text{Jac(\underline{End}}(E))) $.\\ So, $\mathcal{R}(M,M)\sub \text{Jac(\underline{End}}(M))$.\\
 {\bf Case 3:} $M\cong M_1^{r_1}\oplus \ldots\oplus M_q^{r_q}$ with each  $M_i$  indecomposable for all $i=1,\ldots,q$ and $M_i\ncong M_j$ if $i\ne j$ (since $A$ is complete, Krull-Remak-Schmidt (KRS) holds for \underline{CM}$(A)$).\\
 We can assume that $q>1 $ because $q=1$ case follows from case (2). \\
 Now it is sufficient to prove the following claim. \\
 {\bf Claim:} Let $E$ and $L$ be MCM $A$-module. Assume that $E\cong E_1^{a_1}\oplus\ldots\oplus E_n^{a_n}$ and $L\cong L_1^{b_1}\oplus\ldots\oplus L_r^{b_r}$ where  $E_i$   and $L_j$   are distinct indecomposable MCM modules and  $E_i\ncong L_j$ for $i=1,\ldots,n$ and $j=1,\ldots,r$. If the lemma is true for $E$ and $L$, then it is also true for $N=E\oplus L$.\\
 {\bf Proof of the claim:} We know that
 \[ \text{\underline{End}}_A(N)=
 \begin{pmatrix}
 \text{\underline{End}}_A(E) & \text{\underline{Hom}}_A(L,E) \\
 \text{\underline{Hom}}_A(E,L) & \text{\underline{End}}_A(L)
 \end{pmatrix},
 \]

 \[ \text{Jac(\underline{End}}_A(N))=
 \begin{pmatrix}
 \text{Jac(\underline{End}}_A(E)) & \text{\underline{Hom}}_A(L,E) \\
 \text{\underline{Hom}}_A(E,L) & \text{Jac(\underline{End}}_A(L))
 \end{pmatrix}
 \]
 and
 \[ \mathcal{R}(N,N)=
 \begin{pmatrix}
 \mathcal{R}(E,E) & \mathcal{R}(L,E)\\
 \mathcal{R}(E,L) & \mathcal{R}(L,L)
 \end{pmatrix}
 \]	
 Since the result is true for $E$ and $L$, this implies $\mathcal{R}(N,N)\sub \text{Jac(\underline{End}}_A(N)) $.
	\end{proof}
\begin{theorem}\label{isom_th}
Let $(A,\mathfrak{m})$ be a Henselian Gorenstein local ring, $M$ and $N$ be MCM $A$-modules. Then $M\cong N$ in ${\mathcal{D}_A}$ if and only if $M\cong N$ in \normalfont{\underline{CM}}$(A)$.	
\end{theorem}
\begin{proof}
Let $f:M\rt N$ be an isomorphism in \underline{CM}$(A)$. Then $f$ is an isomorphism of $M$ and $N$ in $\mathcal{D}_A$.\\
	For the other direction, suppose $ f:M\rt N$ be an isomorphism in $\mathcal{D}_A$. Then there exists an isomorphism $g:N\rt M$ such that $g\circ f=\mu$ and  $\mu=1 +\delta$ for some $\delta\in \mathcal{R}(M,M)$.
	From Lemma \ref{Jac}, $\delta\in $ Jac(\underline{End}$_A(M))$. This implies $\mu$ is an isomorphism in \underline{CM}$(A)$.
	Therefore, $g\circ f$ is an isomorphism in \underline{CM}$(A)$.
	 Similarly,  $f\circ g$ is also an isomorphism in \underline{CM}$(A)$. This implies $M\cong N$ in \underline{CM}$(A)$.
	\end{proof}

\begin{proposition} \label{ks-body}Let $(A,\mathfrak{m})$ be a Henselian Gorenstein local ring. If $M$ is indecomposabke in $\CMS(A)$ then it is indecomposable in $\mathcal{D}_A$.  Furthermore,
	$\mathcal{D}_A$ is a Krull-Remak-Schmidt (KRS) category.
\end{proposition}
\begin{proof}
	Let $M$ be an MCM $A$-module, then $M\cong M_1^{a_1}\oplus\ldots\oplus M_n^{a_n}$ in $\CMS(A)$ here each  $M_i$  is distinct indecomposable non free MCM $A$-module.
	
	For any indecomposable non free MCM module $N$ we know  $\underline{\text{End}}_A(N)$ is a local ring and End$_{\mathcal{D}_A}(N)= \underline{\text{End}}_A(N)/\mathcal{R}(N,N)$. From lemma \ref{Jac}, $\mathcal{R}(N,N)\sub \text{Jac(\underline{End}}_A(N))$.  So, End$_{\mathcal{D}_A}(N)$ is a local ring. Thus $N$ is indecomposable in $\mathcal{D}_A$.
	\end{proof}
	
	 \bibliographystyle{plain}
	
\end{document}